\documentclass{amsart}
\usepackage{amsthm, amssymb, amsmath, amsfonts}
\usepackage[affil-it]{authblk}
\usepackage{blindtext}
\usepackage{dsfont}
\usepackage{graphicx}
\usepackage[utf8]{inputenc}
\usepackage[english]{babel}
\usepackage[foot]{amsaddr}
\usepackage{multirow}
\usepackage[backend=biber,
style=alphabetic,
sorting=ynt
]{biblatex}
\usepackage[outdir=./]{epstopdf}
\usepackage[margin=1in]{geometry}
\usepackage{verbatim}

\usepackage{tabularx,ragged2e,booktabs,caption}
\newcolumntype{C}[1]{>{\Centering}m{#1}}

\usepackage{setspace}

\newcommand{\bb}{\mathbb}
\newcommand{\floor}[1]{\left\lfloor #1 \right\rfloor}

\bibliography{arxiv}

\newtheorem{theorem}{Theorem}[section]

\newtheorem{lemma}{Lemma}[theorem]
\newtheorem{prop}{Proposition}[section]

\theoremstyle{definition}
\newtheorem{definition}{Definition}[section]

\normalsize
\DeclareGraphicsExtensions{.pdf, .png, .jpg}

\title{Generalized Algorithm for Wythoff's Game with Basis Vector $(2^b,2^b)$}

\author{Shubham Aggarwal, Jared Geller, Shuvom Sadhuka, Max Yu}
\address{PROMYS 2016}
\date{January 2017}

\begin{document}

\maketitle
\begin{abstract}

Wythoff's Game is a variation of Nim in which players may take an equal number of stones from each pile or make valid Nim moves. W. A. Wythoff proved that the set of P-Positions (losing position), $C$, for Wythoff's Game is given by $C := \left\{ (\lfloor k\phi \rfloor, \lfloor k\phi^2 \rfloor), (\lfloor k\phi^2 \rfloor, \lfloor k\phi \rfloor) : k \in \mathbb Z_{\geq 0} \right\}$ \cite{original}. An open Wythoff problem remains where players make the valid Nim moves or remove $kb$ stones from each pile, where $b$ is a fixed integer. We denote this as the $(b,b)$ game. For example, regular Wythoff's Game is just the $(1,1)$ game. In 2009, Duch{\^e}ne and Gravier \cite{Gravier} proved an algorithm to generate the set of P-Positions for the $(2,2)$ game by exploiting the periodic nature of the differences of stones between the two piles modulo $4$. We observe similar cyclic behaviour (see definition \ref{Cyclic}) for any $b$, where $b$ is a power of $2$, modulo $b^2$, and construct an algorithm to generate the set of P-Positions for this game. Let $a$ be a power of $2$. We prove our algorithm works by first showing that it holds for the first $a^2$ terms in the $(a,a)$ game. Next, we construct an ordered multiset for the $(2a,2a)$ game from the $a^2$ terms, and an inductive proof follows. Moreover, we conjecture that all cyclic games require $a$ to be a power of $2$, suggesting that there is no similar structure in the generalised $(b,b)$ game where $b$ isn't a power of $2$. Future directions for generalising this result would likely utilise numeration systems, particularly the PV numbers.

\end{abstract}


\section{Introduction}
Many variations of Wythoff's Game have been explored where different possible moves are permissible. For example, in 2009, Duchene and Gravier solved the Wythoff's Game variation formed by the basis $\left\{(1,0),(0,1),(2,2)\right\}$  \cite{Gravier}. We generalize this result to all \textit{cyclic} games formed by the basis $\left\{(1,0),(0,1),(a,a)\right\}$. In particular, this solves all Wythoff's Game variations where $a=2^i$ for $i\in \mathbb{N}.$ We cannot find another example of a cyclic game, however, and conjecture that no such game exists.

\begin{definition}{(Wythoff's Game)}
A two-player game played with two piles of stones, with $n$ and $m$ stones in each pile respectively, $n, m \in \mathbb{N}$. The players alternate turns. On any given turn, a player may remove $a$ stones from one pile \textit{or} $b$ stones from the other pile \textit{or} an equal number of stones simultaneously from both piles, $a,b\in \mathbb{N}$. The last player to remove a stone wins. We assume throughout this paper that the two players play with optimal strategy.
\end{definition}
\begin{definition}{(N-Position/Hot Position)}
A position in which whichever player's turn it is will win with optimal strategy. For example, $(2,2)$ is a N-position in Wythoff's Game because the player may remove 2 stones from each pile and win. We use the terms \textit{N-Position} and \textit{hot position} interchangeably. 
\end{definition}

\begin{definition}{(P-Position/Cold Position)}
A position in which the player cannot win with optimal strategy. Equivalently, a P-position is any position in which all possible moves will move the game to a N-Position. For example, $(1,2)$ is a cold position since any move will allow the other player to win. We use the terms \textit{P-Position} and \textit{cold position} interchangeably. 
\end{definition}

Plotting the P-positions on $\mathbb{N}^2$ with the x-axis and y-axis each representing a single pile of stones produces an interesting pattern. The P-Positions in the original Wythoff Game cluster around lines with slope $\phi$ and $\frac{1}{\phi}.$ We generalize this notion to $\mathbb{N}^3$ later, with the z-axis representing a third pile of stones. 

\begin{figure}[!ht]
  \centering
      \includegraphics[width=0.5\textwidth]{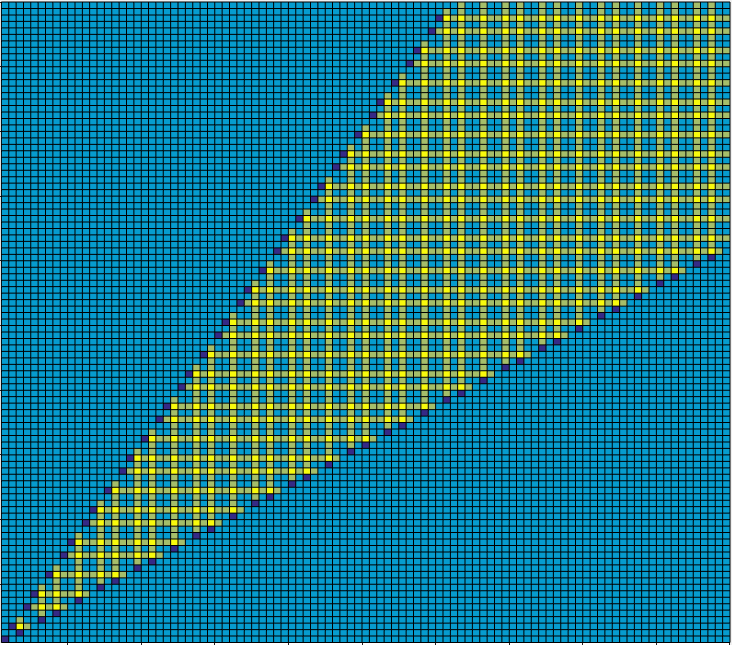}
  \caption{A plot of Wythoff's Game with each axis representing a pile of stones. The P-Positions are denoted in dark blue, while the N-positions are denoted in light blue and yellow. The yellow points show states of the game where there are multiple ways to reach a P-Position.}
\end{figure}

\begin{definition}{$(a,b)$}
     represents the state of the game when pile \textit{A} has \textit{a} stones and pile \textit{B} has \textit{b} stones.
\end{definition}
\begin{definition}{$x(a_1,b_1)$}
 represents the move which changes the state of the game from $(a,b)$ to $(a-xa_1,b-xb_1)$
\end{definition}

\begin{definition}{(Move Vector)}
A game is said to have a \textit{move vector} $(a,b)$ if it is permissible to change the state of the game by the move $\gamma(a,b)$ for some integer $\gamma$.
\end{definition}

\noindent Let $C \subset \mathbb N^2$ be the set of all P-positions for the classical Wythoff's Game. Notice that:

\begin{itemize}
    \item $(a,b)$ and $(b,a)$ represent the same state in the game with just the piles switched. So without loss of generality, we allow $a \le b$.
    \item If $(a,b) \in C$, then for all $i \in \mathbb Z_{\ne 0}, (a+i,b), (a,b+i), (a+i,b+i)\notin C,$ as  playing moves $i(1,0),i(0,1),$ and $i(1,1)$ at each position respectively would move the game to (a,b) which is a cold position. It follows from the definition of a cold position that one cannot get to any other cold position by playing a valid move, a contradiction.
\end{itemize}

\section{Classical Wythoff's Game} 
We now provide the classic result of Wythoff's Game, given first by W.A Wythoff in $1907$ \cite{original}.

\begin{lemma} \label{wyt}
In Wythoff's Game, where the set of moves are generated by the move vectors $\left\{(1,0),(0,1),(1,1)\right\}$, the set of cold positions is given by the following, where $\phi$ is the golden ratio $\frac{1+\sqrt{5}}{2}$: \[\left\{ (\lfloor k\phi \rfloor, \lfloor k\phi^2 \rfloor), (\lfloor k\phi^2 \rfloor, \lfloor k\phi \rfloor): k \in \mathbb Z_{\ge 0} \right\}\]
\end{lemma}

\begin{proof}
It suffices to show that every integer can be expressed as either $\floor{k\phi}$ or $\floor{\ell\phi^2}$ for integers $k,\ell$, and furthermore no two cold positions differ by the move $\gamma(1,1)$ for all $\gamma \in \mathbb{Z}$.
We first show that, for all $k,\ell \in \mathbb{Z}$, $\floor{k\phi} \neq \floor{\ell\phi^2}.$ We proceed by cases. Consider the two cases: $k\phi < k + \ell$ and $k\phi >  k+\ell$.
\begin{description}
    
   \item[Case 1: $k\phi < k + \ell$]
\begin{align*}
    k\phi &< k+\ell \\
    k\phi^2 &< k\phi + \ell\phi \\
    k\phi + k &< k\phi + \ell\phi \\
    k  &< \ell\phi \\
    k+\ell &< \ell\phi+\ell \\
    k+\ell &< \ell(\phi+1) \\
    k+\ell &< \ell\phi^2 \\
\end{align*}
\begin{equation}
k\phi<k+\ell<\ell\phi^2. \label{case1}
\end{equation}

   \item[Case 2: $k\phi > k+\ell$]
\begin{align*}
    k\phi &>  k+\ell \\
    k\phi^2 &>  (k+\ell)\phi \\
    k\phi + k &>  k\phi + \ell\phi \\
    k\phi + k + \ell &>  k\phi + \ell\phi + \ell \\
    k + \ell &>  \ell(\phi + 1) \\
    k + \ell &>  \ell\phi^2
\end{align*}
\begin{equation}
\ell\phi^2 < k+\ell<k\phi \label{case2}
\end{equation}

\end{description}
Because $k+\ell \in \mathbb{N}$, we know that $\floor{k\phi} \neq \floor{l\phi^2}.$ Now suppose $\lfloor k\phi \rfloor + x=\lfloor l\phi \rfloor$ and $\lfloor k\phi^2 \rfloor +x=\lfloor l\phi^2 \rfloor$ for some $x \in \mathbb{Z}.$ Then we have the following:
\begin{align*}
\floor{k\phi^2} - \floor{k\phi} &= \floor{\ell\phi^2} - \floor{\ell\phi} \\
\floor{k\phi +k} -\floor{k\phi} &= \floor{\ell\phi +\ell} - \floor{\ell\phi} \\
\floor{k\phi} +k -\floor{k\phi} &= \floor{\ell\phi} +\ell -\floor{\ell\phi} \\
k &= \ell
\end{align*}

Thus, the move $\gamma(1,1)$ is only permissible for $\gamma=0,$ giving back the same game state. Let $A=\left\{ \lfloor k\phi \rfloor: k \in \mathbb N \right\}, B=\left\{ \lfloor k\phi^2 \rfloor: k \in \mathbb N \right\}.$ It follows that $A\cap B= \emptyset$ and  $A\cup B= \mathbb N$. \\ 

Suppose for sake of contradiction $A\cup B= \mathbb N - S$, for some $S \subset \mathbb N, S\ne \emptyset.$ Let $n\in S.$ We have $ \lfloor k\phi \rfloor \ne n, $ and $ \lfloor l\phi \rfloor \ne n,$ for all $ k\in \mathbb N.$ Let $l\in \mathbb{Z}$ such that $k\phi<n$ and $l\phi^2 < n.$ Then

\[(k+1)\phi>n+1\] and 
\[(\ell+1)\phi^2 <n+1\]
But because $k+l$ is always strictly between $k\phi$ and $\ell\phi^2,$ (see Equations \ref{case1} and \ref{case2}) and $k\phi, \ell\phi^2<n,$
\begin{equation} \label{test}
k+\ell<n
\end{equation}
Similarly, because $(k+1)+(\ell+1)$ lies strictly between $(k+1)\phi$ and $(\ell+1)\phi^2.$\\

\[(\ell+1)\phi^2>n+1\]
\[k+\ell+2>n+1\]
\[k+\ell+1>n\]
\[k+\ell < n < k+\ell+1\]

Since $k+\ell,k+\ell+1$ are consecutive integers, there does not exist an $n\in\bb N$ strictly between them, a contradiction. Hence, because $A\cap B= \emptyset,$ and $S=\emptyset,$ we have $A\cup B= \mathbb N$.

\end{proof}

\begin{theorem}
(Wythoff) For the $k$th cold position $(n_k,m_k)$, with $n \leq m,$ we have:
\begin{align*}
    n_k &= \lfloor k\phi \rfloor = \lfloor m_k\phi \rfloor - m_k \\
    m_k &= \lfloor k\phi^2 \rfloor = \lceil n_k\phi \rceil = n_k + k
\end{align*}
\end{theorem}

\begin{proof}
For $k=1,$ it is simple to check that $(\lfloor k\phi \rfloor, \lfloor k\phi^2 \rfloor)=(2,1)$ and $(\lfloor k\phi^2 \rfloor, \lfloor k\phi \rfloor)=(1,2)$ are indeed cold positions. From Lemma \ref{wyt}, $(\lfloor k\phi \rfloor, \lfloor k\phi^2 \rfloor)$ and $(\lfloor k\phi^2 \rfloor, \lfloor k\phi \rfloor)$ are cold positions for all $k\in \mathbb N.$ Assume for sake of contradiction there exists $a,b \in \mathbb{N}, (a,b)\in C$ that is not generated by the above formula. By Lemma \ref{wyt}, $\lfloor k\phi \rfloor=a$ or $\lfloor k\phi^2 \rfloor=a$ for some $k\in \mathbb N.$ We proceed by cases. 
\begin{description}
    \item[Case 1] $\lfloor k\phi \rfloor=a$

Playing move $(0,\lfloor k\phi^2 \rfloor-b)$ takes $(a,b)$ to $(\lfloor k\phi \rfloor, \lfloor k\phi^2 \rfloor)\in C.$ By the definition of a cold position, $(a,b)\notin C.$

    \item[Case 2] $\lfloor k\phi^2 \rfloor=a$

Playing move $(\lfloor k\phi \rfloor-b,0)$ takes $(a,b)$ to $(\lfloor k\phi^2 \rfloor, \lfloor k\phi \rfloor)\in C.$ By the definition of a cold position, $(a,b)\notin C.$
\end{description}

We have reached a contradiction, so no such cold position $(a,b)$ exists. Thus all cold position are given by $(\lfloor k\phi \rfloor, \lfloor k\phi^2 \rfloor)$ and $(\lfloor k\phi^2 \rfloor, \lfloor k\phi \rfloor),  k\in \mathbb N.$
\end{proof}
It is important to note that this proof strategy cannot be applied to other Wythoff's variations, because it is dependent on the property  $\phi^2=\phi+1,$ which is unique to $\phi.$

\section{Variations of Wythoff's Game}
We now introduce some definitions regarding generalizations of Wythoff's Game.
\begin{definition}{($(a,a)$ Game)}
Let the $(a,a)$ game be defined as the  Wythoff's Game variation formed by the basis vectors $\left \{(1,0),(0,1),(a,a) \right\}$. That is, a player may remove as many stones from one pile or the other, or $k$ stones from each, where $k$ is a positive multiple of $a$.
\end{definition}

Note that regular Wythoff's Game is just the $(1,1)$ game. Just as we can plot Wythoff's game on $\mathbb{N}^2$, plotting the P-Positions for the general $(a,a)$ game produces an interesting graph. Figure $2$ above gives the graphs for the $(1,1)$, $(3,3)$, and $(5,5)$ games.

\begin{figure} 
\includegraphics[width=5cm,height=5cm]{a=1.png}
\hspace{.5cm}
\includegraphics[width=5cm,height=5cm]{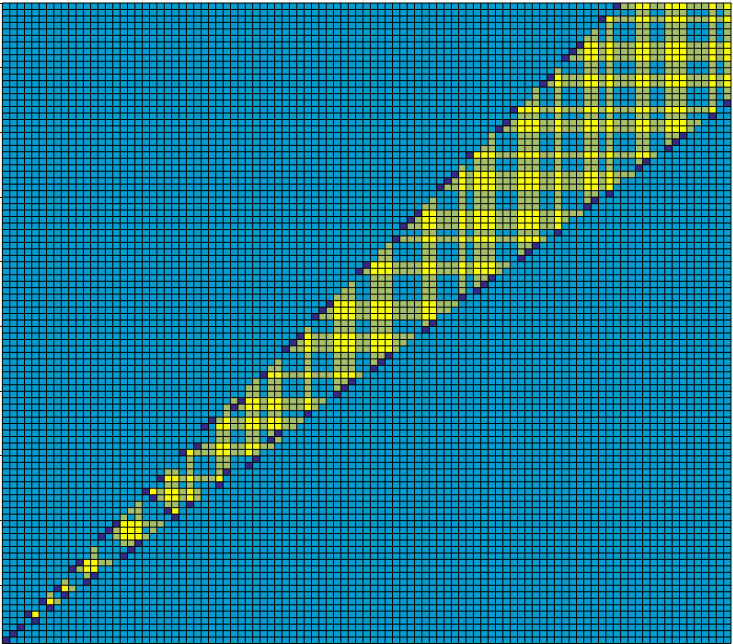}
\hspace{.5cm}
\includegraphics[width=5cm,height=5cm]{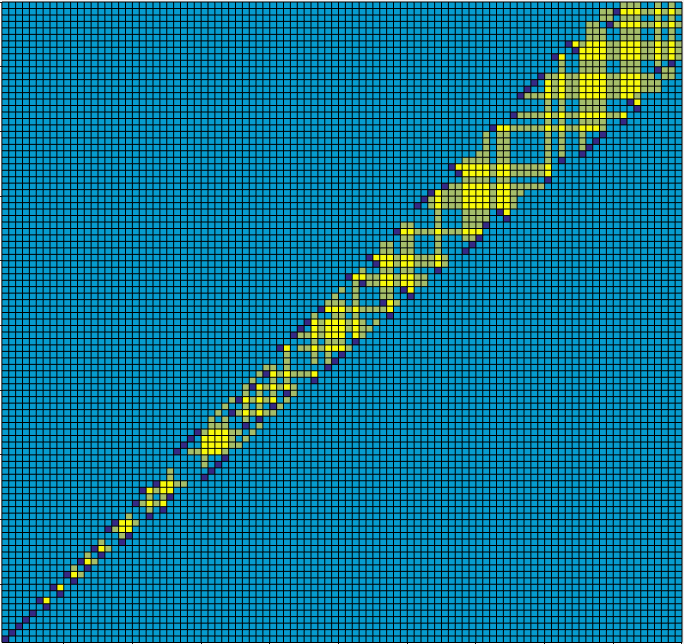}
\caption{A plot of the $(1,1)$, $(3,3)$, and $(5,5)$ games respectively with each axis representing a pile of stones. The P-Positions are denoted in dark blue, while the N-positions are denoted in light blue and yellow. The yellow points show states of the game where there are multiple ways to reach a P-Position. Notice that there are two distinct lines of P-Positions in each game. In \S \ref{JARED} we conjecture the slope of these lines are and $1/\alpha$ wher $\alpha$ is the real root of the quadratic polynomial $\alpha x^2-x-\alpha$. This explains why the slope approaches the line $y=x,$ which is regular Nim (see \S \ref{sancho}).}
\end{figure}

\subsection{Motivating Example}
The $(2,2)$ and $(4,4)$ games provide strong motivation for exploring the $(2^m, 2^m)$ game, which we give a complete characterization for in §\ref{2a}. A table of the $(2,2)$ game P-Positions, as well as the differences between their $x$ and $y$ coordinates, is given below:
\begin{center}
\begin{tabular}{ |c|c|c|c|c|c| } 
\hline
$x_i$ & $y_i$ & $d_i = y_i - x _i$ & $x_i$ & $y_i$ & $d_i = y_i - x _i$\\
\hline
4 & 6 & 2 & 20 & 26 & 6\\ 
5 & 7 & 2 & 21 & 27 & 6\\ 
8 & 11 & 3 & 22 & 29 & 7\\
9 & 10 & 1 & 23 & 28 & 5\\
12 & 16 & 4 & 24 & 32 & 8\\
13 & 17 & 4 & 25 & 33 & 8\\
14 & 19 & 5 & 30 & 39 & 9\\
15 & 18 & 3 & 31 & 38 & 7\\
\hline
\end{tabular}
\end{center}
Notice that for  $(2,2)$ game, $d_i + 2 = d_{i+4}$.  We observe similar behavior for the $(4,4)$ game below, where $d_i + 4 = d_{i+16}$. We will show later that this pattern continues for the general $(2^m,2^m)$ game where $m \in \mathbb{N}$.

\begin{center}
\begin{tabular}{ |c|c|c|c| }
\hline
$x_i$ & $y_i$ & $d_i = y_i - x_i$  \\
\hline
16 & 20 & 4 \\ 
17 & 21 & 4 \\ 
18 & 22 & 4 \\
19 & 23 & 4 \\
24 & 29 & 5 \\
25 & 28 & 3 \\
26 & 31 & 5 \\
27 & 30 & 3 \\
32 & 38 & 6 \\
33 & 39 & 6 \\
34 & 36 & 2 \\
35 & 37 & 2 \\
40 & 47 & 7 \\
41 & 46 & 5 \\
42 & 45 & 3 \\
43 & 44 & 1 \\
\hline
\end{tabular}
\begin{tabular}{ |c|c|c|c| } 
\hline
$x_i$ & $y_i$ & $d_i = y_i - x_i$  \\
\hline
48 & 56 & 8 \\ 
49 & 57 & 8 \\ 
50 & 58 & 8 \\
51 & 59 & 8 \\
52 & 61 & 9 \\
53 & 60 & 7 \\
54 & 63 & 9 \\
55 & 62 & 7 \\
64 & 74 & 10 \\
65 & 75 & 10 \\
66 & 72 & 6 \\
67 & 73 & 6 \\
68 & 79 & 11 \\
69 & 78 & 9 \\
70 & 77 & 7 \\
71 & 76 & 5 \\
\hline
\end{tabular} 
\end{center}

And thus we define a \textit{Cyclic Game} below. We believe this property is unique to the $(a,a)$ game where $a$ is a power of $2$ (see \S\ref{sancho}).  \newline

\begin{definition}{(Cyclic Game)} \label{Cyclic}
    We call the $(a,a)$ game a \textit{Cyclic Game} if $i\equiv j$ $ ($mod $a^2)$ implies $d_{i}\equiv d_{j}  ($mod $a^2)$ where $i,j\in \mathbb Z_{\ge 0}.$  
\end{definition}

\subsection{$(2^m,2^m)$ Game} \label{2a}
We now prove an algorithm that generates the solution set for the Wythoff's variation formed by the basis vectors $\left \{(1,0),(0,1),(2^m,2^m) \right\}$. Duchene and Gravier provided an algorithm for the $m=1$ case, the $(2,2)$ game, by exploiting its cyclic behavior. We provide a similar algorithm for any $(2^m,2^m)$ game. Let $(a_n, b_n)$ be $P$-Positions for the $(2^m,2^m)$ game. For the rest of this section we let $a=2^m$ for some $m\in \mathbb{N}_{\geq 0}$. We give the following definitions.

\begin{definition}{($\left\{d_n\right\}$)}\label{hi}: The sequence of differences such that $d_i = y_i - x_i$, $1 \leq i \leq a^2$ where $(x_i,y_i)$ is a P-position of the $(a,a)$ game. Note that $\left\{d_n\right\}$ has exactly $a^2$ elements. 
  \end{definition}
\begin{definition}{($\left\{d'_n\right\}$)}: The sequence of differences such that $d_i = y_i - x_i$, $1 \leq i \leq 4a^2$ where $(x_i,y_i)$ is a P-position in the $(2a,2a)$ game. Notice that $\left\{d_n\right\}$ has exactly $4a^2$ elements. Moreover, $\left\{d_n'\right\} = \left\{d_n\right\}$ for the $(2a, 2a)$ game.

\end{definition}

  \begin{definition}{($A_i$)}
    $A_i$ denotes the contiguous subsequence of $\left\{d_n\right\}$. That is, $A_i = \left\{d_{ia}, \ldots, d_{ia+a-1} \right\}$. Note that the ordered union of $A_i$ from $i = 0$ to $n$ gives $\left\{d_n\right\}.$ 
  \end{definition}
   
  \begin{definition}{($||$)}
Define $(||)$ to be the concatenation operator. Thus, if we have two sequences, $A=a_1,a_2,\dots,a_n$ and $B=b_1,b_2,\dots,b_n,$ we have $A||B=a_1,a_2,\dots,a_n,b_1,b_2,\dots,b_n.$  
  \end{definition}

   \begin{definition}{($*$)}
Define $(*)$ to be an operator that adds a given integer to each element of a set. Thus, if we the set $A=\left\{a_1,a_2,\dots,a_n\right\},$ we have $A*1=\left\{a_1+1,a_2+1,\dots,a_n+1\right\}.$  
  \end{definition}

  \begin{lemma}\label{Difference Lemma}
  
  Define $B_i$ as the sequence for $\left\{d'_n\right\}$.Then, 
  \begin{align*}
  B_0 = \left\{A_0 || A_0 \right\}, \ldots, B_{a-1} = \left\{A_{a-1}|| A_{a-1}\right\},\ldots, 
  B_{a} = \left\{A_0 * 1 || A_0 *-1 \right\}, \ldots,
  B_{2a-1} = \left\{A_{a-1} * a || A_{a-1} *- a \right\}
  \end{align*}

\begin{proof}
    To prove, we must show that $(i, i+d'_i)$ is a minimal P-position for all $i$. We proceed by induction. It is easy to check that by the difference table for the $(2,2)$ and $(4,4)$ game, that the lemma holds as a base case.

Assume the lemma holds for all $i< m$, where $a=2^m$. Assume that the sequence of differences for the $(a,a)$ game has the following properties (for proofs that these properties then hold for the set of differences of the $(2a,2a)$ game, see \ref{Property 1}, \ref{Power Set}, \ref{p3}, \ref{Property 4}, \ref{Property 5}):

\begin{enumerate}
    \item $k \in \left\{d_n\right\}$ iff $-a<k<a$ \label{1}
    \item For any $i,j$ such that $i \equiv j \mod a$, $d_i \not\equiv d_j \mod a$ \label{2}
    \item For all $0 \leq i < a$, $d_i + d_{a-1} = 0$ \label{3}
    \item For all $i<2a^2,$ $d_i + d_{i+d_i}<a$  \label{4}
    \item $d_i + d_{i-a+d_i} < 2a$ for all $i<2a^2$ \label{5}
\end{enumerate}
We proceed by cases to show that  $\left\{d'_n\right\} = {||}_{j=0}^{2a-1} B_{j}.$
\begin{description}

\item[\textbf{Case 1: $i<2a^2$}]

We proceed via induction to show that that $k=0$ for all $d'_i$. The base case $d'_0=b_0-a_0=0=d_0,$ as $(a_0,b_0)=(0,0)$ holds by definition. Let $d'_g \in \left\{d_n\right\}$, $d'_g = d_x + ak$, and assume $k = 0$ for all $d'_{g}$, $g < i$ as an induction hypothesis. We know that $-a^2< x < a^2$ by property 1. For $k<0$, because $i + d'_i < i$, there exists a P-position $(j, j + d'_j)$ such that $j = i + d'_i < i$. Thus, $(i,i+d'_i)$ and $(j,j+d'_j)$ are the same point. This implies that \\
   \begin{align*}
    i&=j+d'_j\\
    j&=i+d'_i
   \end{align*}
    as $j=i+d'_i<i,$ $d'_j=d_y$ for some $y \in \mathbb{N}$. Simplifying yields 
    \begin{center}
    $d_x+d_y=-ak.$
    \end{center}
    As $d_x,d_y \in  \left\{d_n\right\}$, we have \[-a+1\leq (d_x,d_y)\leq a+1.\] This implies that $k\geq-1$.\\

    If $k=-1$, then we have 
    \begin{center}
    $d_x+d_y=a$
    \end{center}
    and $j=i+d'_i, $ $d'_i=d_x-a$, $d'_j=d_y.$ The definition of  $\left\{d'_n\right\}$ implies for $i<2a^2,$
    \begin{center}
    $y=x-a+d_x$ or $y=x+d_x.$
    \end{center}
    
    Then 
    \begin{center}
    $a=d_x+d_{x+a-d_x}$ or $a=d_x+d_{x+d_x},$
    \end{center}
    which both contradict property \ref{4} and \ref{5} respectively, which we assumed by hypothesis. \\
    
    Thus, $k\geq0.$ To obtain a minimal set of P-Positions, we want to minimize $d'_i$ for all $i.$ Hence, we conclude $k$ must be $0$ to obtain the minimum possible value of $d'_i.$ 
    
\item[\textbf{Case 2:} $ 2a^2\leq i<4a^2$] 

    By the same argument used in Case $1$, $k\geq-1.$ Now we show that $k\neq0.$ \\
    
    Assume for sake of contradiction that $k=0$. Then, for P-Position $(i,i+d'_i),$ $d'_j=d_x$ for some $x$. As $i\geq 2a^2,$ we have 
    \begin{equation} \label{eq4}
        i=2a^2+j, 0\leq j<2a^2.
    \end{equation}
    
    Notice that $d'_j=d_x$ for some $x$ from Case $1$. \\
    
    From (\ref{eq4}) and the definition of $\left\{d'_n\right\}$, $d'_i=d_x+ak$ because $i\equiv j \mod 2a^2$ \\
    
    Thus, the P-Position $(i,i+d'_i)$ can be written as
    \begin{equation} \label{point}
    (j+2^{2n+1},j+2^{2n+1}+d_x). 
    \end{equation}
    
    Because $2a \mid 2a^2,$ we can get to $(j,j+d_x)$ from position (\ref{point}). But $(j,j+d_x)$ is a P-Position, so we get a contradiction. Hence, $k\neq 0.$ \\
    
    By the minimality argument, we choose the least possible value of $k$, namely, $k=-1.$ Consider points $(j,j+d'_j)$ for \[j=2a^2,2a^2+1,\dots,2a^2+a-1.\]
    
    Note $d'_j=d_x-a$ for some $x,$ because $k=-1.$ Moreover, because $d_x \in \left\{d_n\right\},$ $d_x<a,$ we have $d_x-a<0.$ This yields \[j+d'_j=j+d_x-a<j.\]  
    
    Thus, there exists a P-Position, \[(j+d_x-a,j+d_x-a+d'_{j+d_x-a}),\] 
    which is the same as $(j,j+d'_j).$ This implies 
    \begin{center}
    \[j=j+d_x-a+d'_{j+d_x-a}\]
    \end{center}
    \[d_x+d'_{j+d_x-a}=a\]
    \begin{center}
    $d_x+d_y=a$ or $d_x+d_y=2a$
    \end{center}
    because $k=0$ or $k=-1$ for all $d'_k$ with $k<j.$ \\
    
    Both of these equations contradict property \ref{5}. So $k=1$ for \[j=2a^2,2a^2+1,\dots,2a^2+a-1.\]
    
    For \[j=2a^2+a,2a^2+a+1,\dots,2a^2+a+a-1,\]
    as $k=1$ for the previous $a$ values, we can choose $k$ to be minimal i.e. $k=-1.$ \\
    
    Repeating this argument $a$ times yields that the $k$ value alternates between $1$ and $-1$ for blocks $B_{a+i}$ and $B_{a+i+1}$ for all $i, 0\leq i<a-2$. This is exactly what we set out to prove.  
\end{description}
  \end{proof}
   \end{lemma}
   
 We now illustrate how Lemma \ref{Difference Lemma} works for the $(2,2)$ and $(4,4)$ games by listing $\left\{d_n\right\}$ and $\left\{d'_n\right\}$ (which we call the difference tables) of the $(2,2)$ and $(4,4)$ games respectively.  

\begin{center}
\label{tab:title}
\hspace{2cm}
\begin{tabular}{ |c|c|c| } 
\hline
$d_i$ (2,2) & $d'_i$\\
\hline
$d_0$ & 0\\ 
$d_1$ & 0\\ 
$d_2$ & 1\\
$d_3$ & -1\\
\hline
\end{tabular}
\hspace{1.5cm}
\begin{tabular}{|c|c|c|c|c|c|}
\hline
    $d'_j$ (4,4) & $d_i$ (2,2) $+ k$ & $d'_i$ & $d'_j$ (4,4) & $d_i$ (2,2) $+ k$ & $d'_i$ \\
     \hline
     $d'_0$ & $d_0 + 0$ & 0 & $d'_8$ & $d_0 + 2$ & 2\\
     $d'_1$ & $d_1 + 0$ & 0 & $d'_9$ & $d_1 + 2$ & 2\\
     $d'_2$ & $d_0 + 0$ & 0 & $d'_{10}$ & $d_0 - 2$ & -2\\
     $d'_3$ & $d_1 + 0$ & 0 & $d'_{11}$ & $d_1 - 2$ & -2\\
     $d'_4$ & $d_2 + 0$ & 1 & $d'_{12}$ & $d_2 + 2$ & 3\\
     $d'_5$ & $d_3 + 0$ & -1 & $d'_{13}$ & $d_3 + 2$ & 1\\
     $d'_6$ & $d_2 + 0$ & 1 & $d'_{14}$ & $d_2 - 2$ & -1\\
     $d'_7$ & $d_3 + 0$ & -1 & $d'_{15}$ & $d_3 - 2$ & -3\\
     \hline
\end{tabular} \par
\bigskip
Table 1: (2,2) Differences.
\hspace{3cm}
Table 2: (4,4) Differences
\hspace{3cm}
\end{center}

We now prove properties \ref{1}, \ref{2}, \ref{3}, \ref{4}, \ref{5}. 
 
  \begin{lemma} \label{Property 1}
  Property \ref{1}: $k \in \left\{d'_n\right\}$ iff $-2a < k < 2a$.

  \begin{proof}
   
  For the forward direction, by Lemma \ref{Difference Lemma}, $d'_i = d_x + j$, $j = 0$.  By induction hypothesis, $-a < d_x < a$.  It follows that $-2a < d'_i < 2a$.
  And for the reverse direction, notice that
  \begin{align*}
      d_i + 0 & \in \left\{d'_n\right\} \\
      d_i +a & \in \left\{d'_n\right\} \\
      d_i -a & \in \left\{d'_n\right\} \\
  \end{align*}
  for all $i$, $0 \leq i < a$.  Also, for all $k$, $-a < k < a$, $k \in \left\{d_n\right\}$, it follows from the above that for all $k$, $-2a< k < 2a$, $k \in \left\{d'_n\right\}$.
  
  \end{proof}
    \end{lemma}

 \begin{lemma}\label{Power Set}
 Property \ref{2}: If $i \equiv j \pmod{2a}$, then $d'_i\not\equiv d'_j \pmod{2a}$
 \end{lemma}
\begin{proof}
For $i \equiv j \pmod{2a}$, $d'_i \not\equiv d'_j \pmod{a}$, so $d'_i \equiv d'_j \pmod{2a}$ as $i \equiv j \pmod{2a}$ is the same as choosing corresponding terms mod $a$ in $\left\{d_n\right\}$.  Thus, the lemma holds for $\left\{d_n\right\}$. 
\end{proof} 

\begin{lemma}\label{p3}
Property \ref{3}: For all $0 \leq i < a$, $d_i + d_{a-1} = 0$.
\begin{proof}
We show this by induction. Assume true for all values less than or equal to $k$ with set of differences $\left\{d_n\right\}$, then for $2k$, we get $\left\{d'_n\right\}$.  We showed $k = 0, -1, \text{ or } 1$, so if $k=0$, $d'_i + d'_{a-i-1}$ simply follows by assumption.

If $k= 1 \text{ or } -1$, then $d'_i + a + d'_{a-i-1} - a= d'_i + d'_{a-i-1} = 0$
\end{proof}
\end{lemma}

\begin{lemma}\label{Property 4}
Property \ref{4}: For all $j<2a^2$, $d_j + d_{j+d_j}<2a$
\begin{proof}
Because $j<2a^2,$ there exists $d_x,d_y$ for some $x,y$, such that $d'_j=d_x$ and $d_{j-a+d_j}=d_y$\\
By Lemma \ref{Property 1} applied to the $(a,a)$ case, we know that $d_x<a$ and $d_y<a$.Therefore $d_x+d_y=d'_j+d_{j+d_j}<2a$

\end{proof}
\end{lemma}

\begin{lemma}\label{Property 5}
Property \ref{5}:  $d'_j+d_{j-a+d_j}<2a$, for all $j<2a^2$ 
\begin{proof}
Because $j<2a^2,$ there exists $d_x,d_y$ for some $x,y$, such that $d'_j=d_x$ and $d_{j-a+d_j}=d_y.$ By Lemma \ref{Property 1} applied to the $(a,a)$ case, we know that $d_x<a$ and $d_y<a.$ Therefore $d_x+d_y=d'_j+d_{j-a+d_j}<2a$

\end{proof}
\end{lemma}

We now provide definitions and lemmas to allow us to prove Theorem \ref{result}, which allows us to fully characterize the $(a,a)$ game.
\begin{definition}{($P_n$,$Q_n$)}\label{P_n}
  Let $P_n = mex\left\{P_i, Q_i\right\}$ from $i=1$ to  $n-1$ and $Q_n = P_n + a \lfloor \frac{n}{a^2} \rfloor$
\end{definition}

\begin{definition}{$(p_n, q_n)$}\label{p_n}
Let $p_n = P_n$ and $q_n = Q_{n+d_i}$ for $n \equiv j \pmod{a}$ for all $j$ such that $d_j = d_i$
\end{definition}

\begin{lemma} \label{jarediscool}
For all $k \in \mathbb{Z}_{\geq 0}$, there exist integers $k_0, k_1, \ldots, k_{a-1}$ such that 
\begin{align*}
(p_{a^2k + 0}, q_{a^2k + 0}) &= (ak_0 + 0, ak_0 + 0 + ak + d_0) \\
  & \vdots \\
  (p_{a^2k + a-1}, q_{a^2k + a-1}) &= (ak_0 + a-1, ak_0 + a-1 + ak + d_{a-1}) \\
  (p_{a^2k + a}, q_{a^2k + a}) &= (ak_1 + 0, ak_1 + 0 + ak + d_a) \\
  & \vdots \\
  (p_{a^2k + a^2 -1}, q_{a^2k + a^2 - 1}) &= (ak_{a-1} + a - 1, ak_{a-1} + a - 1 + ak + d_{a^2-1})
\end{align*}
\begin{proof}
By definition of $(p_n, q_n)$ and $(P_n, Q_n)$, each point $(p_{a^2k+i},q_{a^2k+i} )$ satisfies $q_{a^2k+i} - p_{a^2k+i} = d_{a^2k+i}$.  So it suffices to show that $p_{a^2k+i}$ is correct for all $i$. We proceed by induction. For a base case $k=0,$ we have $p_i = i$, hence the lemma follows. Assume the lemma is true for all values up to $k$ as an inductive hypothesis.\\
Then, as $p_n$ is an increasing sequence, $p_{a^2k + a} > ak_{a-1} + a-1$.  By definition of $p_n$ (\ref{p_n}) using the mex properties, $p_{a^2k+a}$ cannot equal $ak_{a-1} + a$ iff $ak_{a-1} + a$ has occurred in the $q_n$ sequence before.  By induction hypothesis,
\begin{center}
$ak_{a-1} + a = q_{a^2k +ai +i}$, some $i, 0 \leq i < a$
\end{center}
as $ak_{a-1} + a$ has occurred before. By the induction hypothesis, all terms of the form $ak_{a-1} + a+j$, $0 \leq j < a$ have occurred before in the $q_n$ sequence. Because only a finite number of terms have occurred before, we can choose $k'_0$ such that $p_{a^{2}(k+1)} = ak'_0 + l$. If $ l \neq 0$, then $ak'_0 - l \equiv 0 \pmod{a}$, which implies $ak_0 + l$ has occurred in the sequence before, a contradiction.

Hence, $p_{a^{2}(k+1)} = ak'_0$, which implies the lemma holds for $p_{a^{2}(k+1)}$ to $p_{a^{2}(k+1) + a-1}$.\\
Repeating the same argument for $p_{a^{2}(k+1) + a}$ to $p_{a^{2}(k+1) + 2a-1}$, \ldots $p_{a^{2}(k+1) + a(a-1)}$  to $p_{a^{2}(k+1) + a^2-1}$ shows the lemma is true for $k+1$.  This completes the proof by induction.

\end{proof}
\end{lemma}

\begin{theorem}\label{result}
The sequence $(p_n, q_n)$ describes all P-positions of the $(a, a)$ game.

\begin{proof}
We know 
\begin{align*}
  (p_{a^2k + 0}, q_{a^2k + 0}) &= (ak_0 + 0, ak_0 + 0 + ak + d_0) \\
  & \vdots \\
  (p_{a^2k + a-1}, q_{a^2k + a-1}) &= (ak_0 + a-1, ak_0 + a-1 + ak + d_{a-1}) \\
  (p_{a^2k + a}, q_{a^2k + a}) &= (ak_1 + 0, ak_1 + 0 + ak + d_a) \\
  & \vdots \\
  (p_{a^2k + a^2 -1}, q_{a^2k + a^2 - 1}) &= (ak_{a-1} + a - 1, ak_{a-1} + a - 1 + ak + d_{a^2-1})
\end{align*} 
by Lemma \ref{jarediscool}.
The base case $k = 0$ was proven in Lemma \ref{Difference Lemma}.  Assume the proposition holds true for $p_{a^2i-1}, q_{a^2i-1}$ for all $i<k$.We wish to show the lemma holds for $p_{a^2k}, p_{a^2k + 1}, \ldots, p_{a^2k+a^2-1}$. Suppose $p_{a^2k+e} \equiv p_{a^2l + f} \pmod{a}$ for $e,f \in \mathbb{Z}_{\geq0}$.  Then,
\begin{align*}
    p_{a^2k + e} - q_{a^2k + e} &= p_{a^2l + f} - q_{a^2l + f} \\
    ak + d_e &= al + d_f \\
    k-l &= \frac{d_f-d_e}{a}. \\
\end{align*}
And, since $\frac{d_f-d_e}{a} \leq \frac{a-1-(-(a-1))}{a}$, which is the maximal value $\frac{d_f-d_e}{a}$, it follows that $k-l = 1$ as $k-l\leq\frac{a-1-(-(a-1))}{a} = 2 - \frac{2}{a} < 2$ and $k-l \in \mathbb{N}$.
Then, 
\begin{align*}
    ak + d_e &= al + d_f \\
    a(l+1) + d_e &= al + d_f \\
    a + d_e &= d_f, \\
\end{align*}
which contradicts lemma \ref{Power Set}, completing our proof.

\end{proof}
\end{theorem}

\subsection{$(b,b)$ Game}
We now introduce a general algorithm for the $(b,b)$ game, where $b \in \mathbb{N}.$ However, this algorithm does not provide an \textit{explicit} representation of the set of P-Positions for the general $(b,b)$ game.
\begin{theorem}
   Define $(0, 0)$ to be a P-Position for all games. Let $(a_0, b_0) = (0, 0)$ and define $(a_n, b_n)$ to be the $nth$ P-Position with $a_n < b_n$.  Then, $a_n = mex \left\{a_i , b_i : 0 \leq i < n\right\},$ where
   $b_n$ is the smallest number not in $\left\{a_i , b_i : 0 \leq i < n\right\}$ such that for any $i < n$ the following cannot be simultaneously true:
   \begin{center}
   $b_n-a_n=b_i-a_i$\\
        $a_n \equiv a_i \pmod{a}$\\
   \end{center}

\begin{proof}
First, we show that if the algorithm generates $(a_n, b_n)$, then there are no valid moves such that $(a_n, b_n)$ goes to $(a_i, b_i)$ for $i < n$.  Let $(a_i, b_i)$, $(a_j, b_j)$ be two points generated by the algorithm, $i > j$.  Assume for the sake of contradiction that there exists a valid move such that $(a_i, b_i)$ is sent to $(a_j, b_j)$.  The mex property in the algorithm ensures that there is exactly one P-position in every row and column.  Thus, $a_i \neq a_j$ and $b_i \neq b_j$. Thus, their must be a valid move $k(a,a)$ with $k \geq 1$.  Then, we get 
\begin{align*}
a_j &= a_i - ka \\
b_j &= b_i - ka
\end{align*}
This implies $a_i \equiv a_j \pmod{a}$ and $a_j - b_j = a_i - b_i$.  Hence, $(a_i, b_i)$ couldn't have been chosen by the algorithm, a contradiction.\\
Now we show that any P-position is always chosen by the algorithm.  Once a P-position $(x,y)$, we mark all lattice points $(x,k)$ for any $k>y$, $(l,y)$ for any $l>x$, and $(x+ka, y+ka)$ for any $k \geq 1$ as a N-Position.  Assume for the sake of contradiction that there exists a lattice point not marked as an N-Position or P-position.  Choose some point $(x,y)$ such that $(x,y)$ is the closest to the origin and not marked.  Without loss of generality, let $x \leq y$.  Since $(x,y)$ is not marked, and every column has exactly $1$ P-position, we can find a P-position $(x,y')$, $y'>y$.  Thus, either $y'=y$ which is not true by assumption, or there exists $(x_0, y_0)$ such that $x-x_0 \equiv y-y_0 \pmod{a}$.   But then $(x,y)$ is marked as a P-position by this algorithm, a contradiction.
\end{proof}
\end{theorem}

\subsection{Asymptotic Behavior}\label{JARED}
As we increase $a$ in the Wythoff variation formed by the basis $\left\{(1,0),(0,1),(a,a)\right\}$, the lines that arise from a plot of the P-Positions on $\mathbb{N}^2$ seem to converge. We conjecture that as $a$ grows increasingly large, this game will approach the standard Nim game, with cold positions along the line $y=x$. 

This behavior intuitively can be explained by setting the third basis vector $(a,a)$ to a point defined at infinity, $(\infty,\infty)$. Thus, this third vector has no effect on the graph within the grid $(\infty,\infty)$, equivalent to not existing. Therefore, we have move vectors $\left\{(1,0),(0,1)\right\},$ which is just the standard Nim game. 

\subsection{N-Dimensional Games}

\begin{prop}
Consider a game with $n$ piles, and move vectors $v_1, v_2, \ldots, v_m$ where $v_1, v_2, \ldots, v_m \in \mathbb Z ^n.$ If the move vectors and are linearly independent, this game is equivalent to $n$ pile Nim.
\end{prop}
\begin{proof}
Let $P$ be a position in the game with $P = (p_1, p_2, \ldots, p_n)$.  Then $P$ can be written as a unique canonical decomposition in terms of the move vectors.
\begin{center}
$P = a_1 v_1 + a_2 v_2 + \cdots + a_m v_m$
\end{center}
Playing a move is $k v_i$ is equivalent to reducing $a_i$ by $k$. Moves are played until we reach the $0$ vector, which is the same as playing $m$-pile nim. 
\end{proof}

\section{Further Investigations} \label{sancho}
The cold positions of the $(a,a)$ game seem to asymptotically converge towards the lines with slopes $\alpha$ and $1/\alpha$ where
$\alpha$ is the positive real root of the quadratic $ax^2-x-a$.
A similar result was proved by Fraenkel for a different variation of Wythoff's Game \cite{Fraenkel}. If proven true, this result could help resolve other open Wythoff problems. We also make the following conjectures:
\begin{figure}[h] 
\includegraphics[width=7.3cm,height=7.3cm]{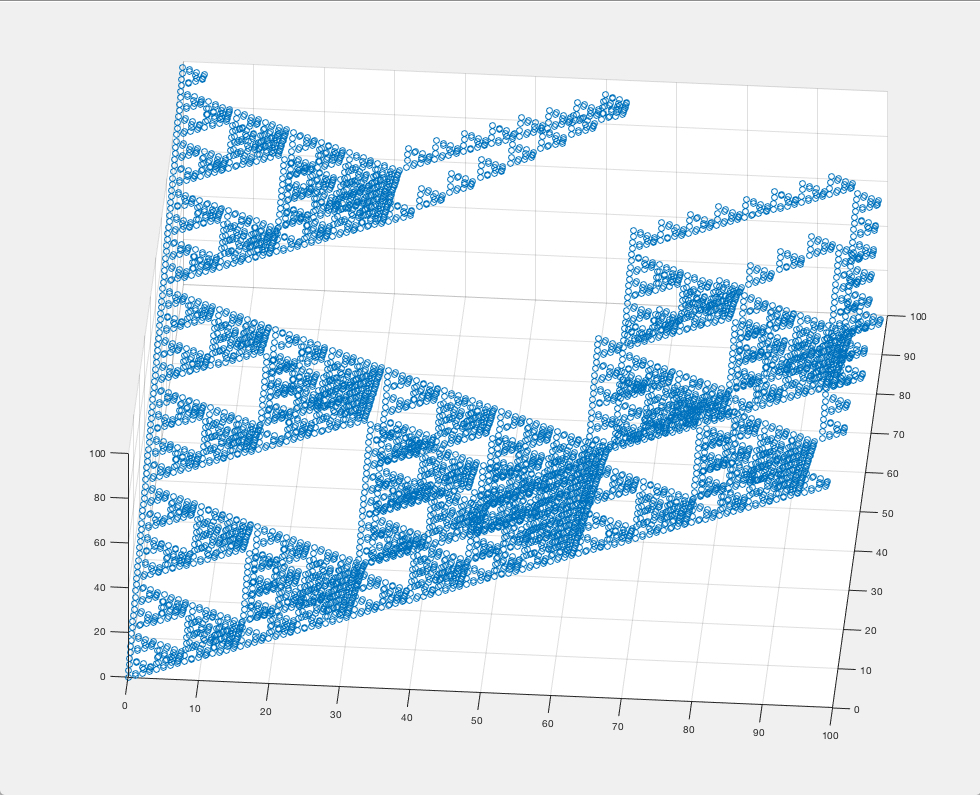}
\hspace{1.5cm}
\includegraphics[width=7.3cm,height=7.3cm]{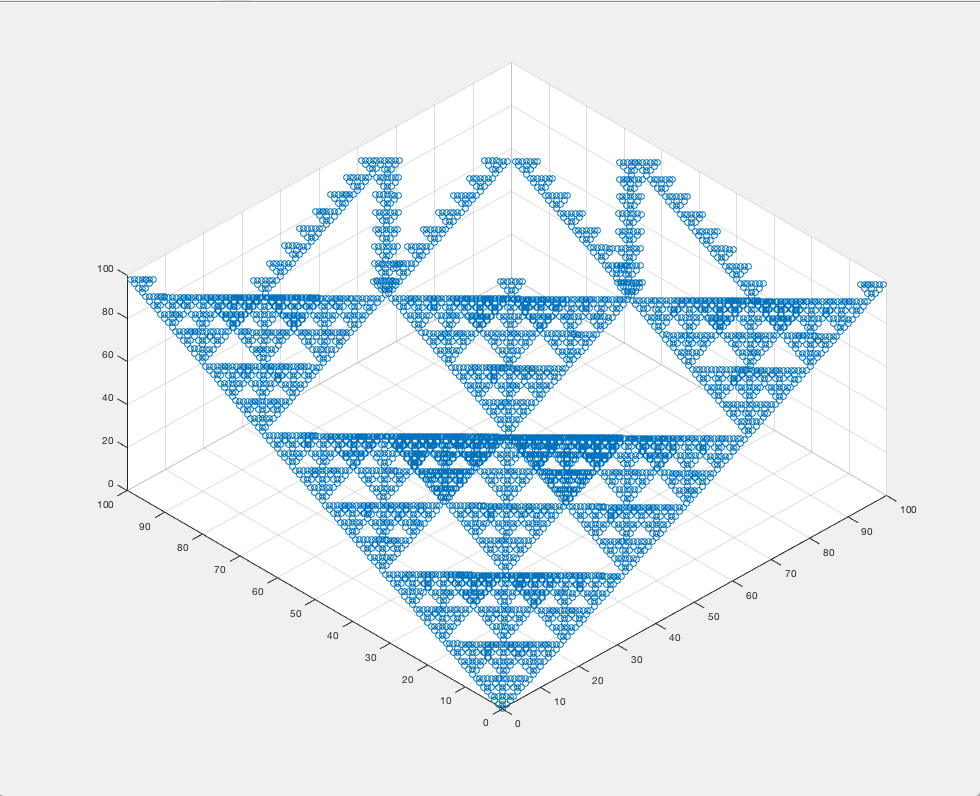} \label{abc}
\caption{Two different angles of the standard Wythoff's Game formed with $3$ piles and the basis $\left\{(1,0,0),(0,1,0),(0,0,1),(1,1,1)\right\}$. Notice that this generates a Sierpinski Sponge. There seem to be clear analogs between this and the $2$-pile game.}
\end{figure}

\begin{itemize}
    \item (Sierpinski Sponge Conjecture) The set of P-Positions of the $3$ dimensional Wythoff's Game (Wythoff's Game with $3$ piles) generates the Sierpinski Sponge as shown above. The move vectors for this game are $\left\{(1,0,0),(0,1,0),(0,0,1),(1,1,1)\right\}$. 
    \item (Grid Conjecture) Let $(x,y)$ be a P-Position in the $(a,a)$ game. We conjecture that there are exactly $a^2$ P-positions satisfying $x,y \in [0, a^2-1]$ in this game. In other words the first $a^2$ P-positions lie within the square grid $0$ to $a^2-1$. 
    \item (Asymptote Conjecture) The asymptotic slope  of the $(a,a)$ game is $\alpha=\frac{\sqrt{4a^2+1}+1}{2a}$, the positive real solution to the quadratic $\alpha^2-\alpha-1=0.$
    \item (Cyclic Games Conjecture) The  $(a,a)$ game is cylic iff $a=2^m$, $m \in \mathbb{Z}^+$. Proving this would show that our solution to the $(a,a)$ game cannot be applied to the general $(b,b)$ game.
\end{itemize}

\section{Acknowledgements}
We would like to extend our thanks to Professor Paul Gunnells (UMass Amherst) for proposing this project, Professor A.S. Fraenkel (Weizmann Institute of Science) for helping direct us to relevant work related to Wythoff's Game Variations, Tim Ratigan for supervising us throughout this research, and the PROMYS Program and Clay Mathematics Institute for sponsoring this research.

\printbibliography
\end{document}